\documentclass[graybox]{svmult}

 \usepackage{mathptmx}       
 \usepackage{helvet}         
 \usepackage{courier}        
 \usepackage{type1cm}        
 \usepackage{makeidx}         
 \usepackage[bottom]{footmisc}
 \makeindex             
\usepackage{amsmath}
\usepackage{amsfonts,amssymb,fontenc, mathrsfs, mathtools}
\usepackage{theorem}

\usepackage{tikz}
\usepgflibrary{arrows}
\usetikzlibrary{matrix}

\usepackage[all]{xy}

\input amssym.def

\numberwithin{equation}{section}

\numberwithin{figure}{section}

\makeatletter
\newcommand\ourqedold{\relax\ifmmode\Box\else
{\unskip\nobreak\hfil\penalty50\hskip1em\null\nobreak\hfil\qedsymbol
\parfillskip=\z@\finalhyphendemerits=0\endgraf}\fi}

\newcommand{\bfzero}{{\bf 0}}

\newcommand{\homP}{{\mathsf{HomP}}}

\newcommand{\op }[1]{\mathsf{#1}}

\newcommand{\Tree}{{\mathsf{T r e e}}}

\newcommand{\Lie}{{\mathsf{Lie}}}

\newcommand{\Ass}{{\mathsf{As}}}

\newcommand{\Cobar}{{\mathrm{ C o b a r}}}

\newcommand{\Conv}{{\mathrm{Conv}}}

\newcommand{\dgSh}{{\mathsf{dgSh}}}

\renewcommand{\c}{{\circ}}

\newcommand{\mSH}{{\mathfrak{SH}}}

\newcommand{\Def}{{\rm D e f }}
\newcommand{\End}{{\mathsf {E n d} }}
\newcommand{\Hom}{{\mathrm {Hom}}}

\newcommand{\TS}{{\mathrm{TS}}}

\newcommand{\Cyl}{{\mathrm {Cyl}}}

\newcommand{\Isom}{{\mathsf{Isom}}}

\newcommand{\coDer}{{\mathrm {coDer}}}

\newcommand{\id}{{\mathsf{ i d} }}

\newcommand{\Sh}{{\mathrm {S h} }}

\newcommand{\wt}[1]{{\widetilde{#1}}}

\newcommand{\al}{{\alpha}}

\newcommand{\la}{{\lambda}}

\newcommand{\mU}{{\mathfrak{U}}}
\newcommand{\mV}{{\mathfrak{V}}}

\newcommand{\si}{{\sigma}}

\newcommand{\cF}{{\mathcal{F}}}
\newcommand{\pa}{{\partial}}

\newcommand{\bsi}{{\bf s}^{-1}\,}

\newcommand{\bs}{{\bf s}}
\newcommand{\bt}{{\bf t}}

\newcommand{\cQ}{{\mathcal Q}}

\newcommand{\cD}{{\cal D}}
\newcommand{\cA}{{\mathcal{A}}}
\newcommand{\cG}{{\mathcal{G}}}

\newcommand{\cB}{{\mathcal{B}}}

\newcommand{\cV}{{\cal V}}

\newcommand{\bbK}{{\mathbb K}}

\newcommand{\bbZ}{{\mathbb Z}}

\newcommand{\pf}{{\,\pitchfork}}

\date{}

\newtheorem{thm}{Theorem}[section]
\newtheorem{defi}[thm]{Definition}

\newtheorem{cor}[thm]{Corollary}

\newtheorem{prop}[thm]{Proposition}
\newtheorem{cond}[thm]{Condition}

\theorembodyfont{\rm}

\title*{The deformation complex is a homotopy invariant of a homotopy algebra}

\author{Vasily Dolgushev and Thomas Willwacher}
\date{}

\institute{Vasily Dolgushev \at Department of Mathematics,
Temple University, 
Wachman Hall Rm. 638,
1805 N. Broad St.,
Philadelphia PA, 19122 USA, \email{vald@temple.edu}
\and Thomas Willwacher \at Department of Mathematics, 
ETH Z\"urich, 
R\"amistrasse 101, 8092 Z\"urich, Switzerland,
\email{thomas.willwacher@math.ethz.ch}}

\begin{document}


\maketitle

~\\
\begin{flushright}
{\small
{\it ``Well, mathematician X likes to write formulas...''}\\ Alexander Beilinson}
\end{flushright}
~\\

\abstract*{
To a homotopy algebra one may associate its deformation complex, 
which is naturally a differential graded Lie algebra. We show that 
$\infty$-quasi-isomorphic homotopy algebras have $L_\infty$-quasi-isomorphic
deformation complexes by an explicit construction. 
}

\abstract{
To a homotopy algebra one may associate its deformation complex, 
which is naturally a differential graded Lie algebra. We show that 
$\infty$-quasi-isomorphic homotopy algebras have $L_\infty$-quasi-isomorphic
deformation complexes by an explicit construction. 
}

\section{Introduction}
Given two homotopy algebras $\cA$, $\cB$ of a certain type (e. g. $L_\infty$- or $A_\infty$-algebras), we may 
define their deformation complexes $\Def(\cA)$ and $\Def(\cB)$, which are differential graded Lie algebras. 
Suppose that $\cA$ and $\cB$ are quasi-isomorphic. For example, there may be an $L_\infty$- or $A_\infty$-quasi-isomorphism $\cA\to\cB$. 
It is natural to ask whether in this case the deformation complexes $\Def(\cA)$ and $\Def(\cB)$ are 
quasi-isomorphic as $L_\infty$-algebras, and whether a quasi-isomorphism may be written down in 
a (sufficiently) functorial way. The answer to the above question is (not surprisingly) yes, as is probably known to the experts.
However, the authors were not able to find a proof of this statement in the literature in the desired generality.  

The modest purpose of this note is fill in this gap by presenting the construction of an explicit 
sequence of quasi-isomorphisms connecting $\Def(\cA)$ with $\Def(\cB)$\,.

This note is organized as follows. After a brief description of our construction, we recall, in Section \ref{sec:prelim}, the 
necessary prerequisites about homotopy algebras. Section \ref{sec:main} is the core of this paper. In this section, we 
formulate the main statement (see Theorem \ref{thm:main}), describe various auxiliary constructions, and finally prove
Theorem \ref{thm:main} in Subsection \ref{sec:proof}. Section \ref{sec:sheaves} is devoted to the notion of 
homotopy algebra and its deformation complex in the setting of dg sheaves on a topological space.
In this section, we give a version of Theorem \ref{thm:main} (see Corollary \ref{cor:main}) and 
describe its application.

\subsection{The construction in a nutshell}
For the reader who already knows some homotopy algebra, here is what we will do in this note. First, the homotopy algebras of the type we consider are governed by some operad $\op P$. For example, for $A_\infty$-algebras $\op P=\Ass_\infty$ and for $L_\infty$-algebras $\op P=\Lie_\infty$. Providing $\op P$ algebra structures on $\cA$ and $\cB$ is equivalent to providing operad maps $\op P\to \End(\cA)$, $\op P\to \End(\cB)$ into the endomorphism operads. The deformation complexes $\Def(\cA)$, $\Def(\cB)$ are by definition the deformation complexes of the operad maps $\Def(\cA)=\Def(\op P\to \End(\cA))$, $\Def(\cB)=\Def(\op P\to \End(\cB))$. 

Similarly, one may define a two-colored operad $\homP$, whose algebras are triples $(\cA, \cB, F)$, where 
$\cA$ and $\cB$ are $\op P$ algebras and $F$ is a homotopy ($\infty$-)morphism between them.
Furthermore, given an $\infty$ quasi-isomorphism $\cA \leadsto \cB$, 
we may build a colored operad map $\homP\to \End(\cA,\cB)$ into the colored endomorphism operad. One may build a deformation complex $\Def(\homP\to \End(\cA,\cB))$, which is an $L_\infty$-algebra. Furthermore, there are natural maps
\[
 \Def(\cA) \leftarrow \Def(\homP\to \End(\cA,\cB)) \to \Def(\cB)
\]
which one may check to be quasi-isomorphisms. Hence this zigzag constitutes desired 
explicit and natural quasi-isomorphisms of $L_\infty$-algebras.

~\\
\noindent
\textbf{Acknowledgements:} We would like to thank Bruno Vallette for useful discussions.
V.D. acknowledges the NSF grant DMS-1161867 and 
the grant FASI RF 14.740.11.0347. 
T.W. thanks the Harvard Society of Fellows and the Swiss National Science Foundation (grant PDAMP2\_137151) for their support.

\section{Preliminaries}
\label{sec:prelim}

The base field $\bbK$ has characteristic zero. The underlying symmetric 
monoidal category is the category of unbounded cochain complexes of 
$\bbK$-vector spaces. We will use the notation and conventions about labeled planar 
trees from \cite{notes}. In particular, we denote by $\Tree(n)$ the groupoid of 
$n$-labeled planar trees. As in \cite{notes}, we denote by $\Tree_2(n)$ 
the full subcategory of $\Tree(n)$ whose objects are $n$-labeled planar 
trees with exactly $2$ nodal vertices. For a groupoid $\cG$, the notation 
$\pi_0(\cG)$ is reserved for the set of isomorphism classes of objects 
in $\cG$. 

We say that an $n$-labeled planar tree $\bt$ is a {\it pitchfork} if 
each leaf of $\bt$ has height\footnote{Recall that the height of a vertex $v$ is the 
length of the (unique) path which connects $v$ to the root vertex.} 3.
Figure \ref{fig:pf} shows a pitchfork while 
figure  \ref{fig:non-pf}  shows a tree that is not a pitchfork.  
\begin{figure}[htp] 
\begin{minipage}[t]{0.45\linewidth}
\centering 
\begin{tikzpicture}[scale=0.5]
\tikzstyle{w}=[circle, draw, minimum size=3, inner sep=1]
\tikzstyle{b}=[circle, draw, fill, minimum size=3, inner sep=1]
\node[b] (r) at (2, 0) {};
\node[w] (v1) at (2, 1) {};
\node[w] (v2) at (1, 2) {};
\node[b] (v3) at (0.5, 3) {};
\draw (0.5,3.5) node[anchor=center] {{\small $2$}};
\node[b] (v4) at (1.2, 3) {};
\draw (1.2,3.5) node[anchor=center] {{\small $1$}};
\node[w] (v5) at (2, 2) {};
\node[b] (v6) at (2, 3) {};
\draw (2,3.5) node[anchor=center] {{\small $4$}};
\node[w] (v7) at (3, 2) {};
\node[b] (v8) at (3, 3) {};
\draw (3,3.5) node[anchor=center] {{\small $3$}};
\draw (r) edge (v1);
\draw (v1) edge (v2);
\draw (v1) edge (v5);
\draw (v1) edge (v7);
\draw (v2) edge (v3);
\draw (v2) edge (v4);
\draw (v5) edge (v6);
\draw (v7) edge (v8);
\end{tikzpicture}
\caption{An example of a pitchfork} \label{fig:pf}
\end{minipage}
\begin{minipage}[t]{0.45\linewidth}
\centering 
\begin{tikzpicture}[scale=0.5]
\tikzstyle{w}=[circle, draw, minimum size=3, inner sep=1]
\tikzstyle{b}=[circle, draw, fill, minimum size=3, inner sep=1]
\node[b] (r) at (2, 0) {};
\node[w] (v1) at (2, 1) {};
\node[w] (v2) at (1.5, 2) {};
\node[b] (v3) at (1, 3) {};
\draw (1,3.5) node[anchor=center] {{\small $2$}};
\node[b] (v4) at (2, 3) {};
\draw (2,3.5) node[anchor=center] {{\small $1$}};
\node[b] (v5) at (3, 2) {};
\draw (3,2.5) node[anchor=center] {{\small $3$}};
\draw (r) edge (v1);
\draw (v1) edge (v2);
\draw (v2) edge (v3);
\draw (v2) edge (v4);
\draw (v1) edge (v5);
\end{tikzpicture}
\caption{This is not a pitchfork} \label{fig:non-pf}
\end{minipage}
\end{figure}

The notation $\Tree_{\pf}(n)$ is reserved for the full sub-groupoid 
of $\Tree(n)$ whose objects are pitchforks.

Let $C$ be a coaugmented dg cooperad satisfying the following technical condition:
\begin{cond}
\label{cond:C}
The cokernel $C_{\c}$ of the coaugmentation 
carries an ascending exhaustive filtration
\begin{equation}
\label{C-circ-filtr}
\bfzero  = \cF^0 C_{\c} \subset  \cF^1 C_{\c} \subset 
  \cF^2 C_{\c} \subset \dots 
\end{equation}
which is compatible with the pseudo-cooperad structure on $C_{\c}$.
\end{cond}

For example, if the dg cooperad $C$ has the properties
\begin{equation}
\label{C-conditions}
C(1) \cong \bbK, \qquad \qquad C(0) = \bfzero
\end{equation} 
then the filtration ``by arity minus one'' on $C_{\c}$ satisfies the above technical condition.

For a cochain complex $\cV$ we denote by 
\begin{equation}
\label{C-cV}
C(\cV) := \bigoplus_{n \ge 1} \Big( C(n) \otimes \cV^{\otimes\, n}  \Big)_{S_n}
\end{equation}
the ``cofree'' $C$-coalgebra co-generated by $\cV$\,.

We denote by 
\begin{equation}
\label{coDer-C-cV}
\coDer \big(C(\cV)\big)
\end{equation}
the cochain complex of coderivations of the cofree coalgebra 
$C(\cV)$ co-generated by $\cV$\,. In other words, $\coDer \big(C(\cV)\big)$
consists of $\bbK$-linear maps 
\begin{equation}
\label{cD-coder}
\cD : C(\cV) \to C(\cV)
\end{equation}
which are compatible with the $C$-coalgebra structure on $C(\cV)$ in the 
following sense: 
\begin{equation}
\label{coder-axiom}
\Delta_n \circ \cD  = \sum_{i=1}^n \big( \id_C \otimes \id_{\cV}^{\otimes (i-1)} \otimes \cD \otimes 
\id_{\cV}^{n-i} \big) \circ \Delta_n
\end{equation}
where $\Delta_n$ is the comultiplication map 
$$
\Delta_n :  C(\cV) \to \Big( C(n) \otimes  \big( C(\cV) \big)^{\otimes n} \Big)_{S_n}\,.
$$ 
The $\bbZ$-graded vector space \eqref{coDer-C-cV} carries a natural differential
$\pa$ induced by those on $C$ and $\cV$\,.

Since the commutator of two coderivations is again a coderivation, the 
cochain complex \eqref{coDer-C-cV} is naturally a dg Lie algebra.

Recall that, since the $C$-coalgebra $C(\cV)$ is cofree, every
coderivation $\cD : C(\cV) \to C(\cV) $ is uniquely determined by 
its composition $p_{\cV} \circ \cD$ with the canonical projection:
\begin{equation}
\label{p-cV}
p_{\cV} : C(\cV) \to \cV\,.
\end{equation}

We denote by 
\begin{equation}
\label{coDer-pr-C-cV}
\coDer' \big(C(\cV)\big)
\end{equation}
the dg Lie subalgebra of coderivations $\cD \in \coDer \big(C(\cV)\big)$
satisfying the additional technical condition
\begin{equation}
\label{cD-cond}
\cD \Big|_{\cV} = 0\,.
\end{equation}

Due to \cite[Proposition 4.2]{notes}, the map 
$$
\cD \mapsto p_{\cV} \circ \cD
$$
induces an isomorphism of dg Lie algebras
\begin{equation}
\label{coDer-pr-Conv}
\coDer' \big(C(\cV) \big) \cong \Conv(C_{\circ}, \End_{\cV})\,,
\end{equation}
where the differential $\pa$ on $\Conv(C_{\circ}, \End_{\cV})$ comes solely 
from the differential on $C_{\circ}$ and $\cV$\,. 
Here $\Conv(\cdots)$ denotes the convolution Lie algebra of ($\mathbb{S}$-module-)maps from a cooperad to an operad, cf. \cite[section 6.4.4]{LV-book}.

Recall that \cite[Proposition 5.2]{notes} 
$\Cobar(C)$-algebra structures on a cochain complex $\cV$
are in bijection with Maurer-Cartan (MC) elements in $\coDer' \big(C(\cV) \big)$, i.~e., with degree $1$ coderivations 
\begin{equation}
\label{Q}
Q \in  \coDer' \big(C(\cV) \big)
\end{equation}
satisfying the Maurer-Cartan equation 
\begin{equation}
\label{MC-Q}
\pa Q + \frac{1}{2}[Q, Q] = 0\,.
\end{equation}

Hence, given a $\Cobar(C)$-algebra structure on $\cV$, we may consider 
the dg Lie algebra \eqref{coDer-pr-Conv} and the $C$-coalgebra $C(\cV)$
with the new differentials 
\begin{equation}
\label{diff-tw-Q}
\pa + [Q, ~]\,,
\end{equation}
and
\begin{equation}
\label{diff-tw-Q1}
\pa + Q\,,
\end{equation}
respectively.

In this text we use the following ``pedestrian'' definition of homotopy 
algebras: 
\begin{defi}
\label{dfn:homot-alg}
Let $C$ be a coaugmented dg cooperad satisfying Condition \ref{cond:C}. 
\emph{A homotopy algebra of type} $C$ is a $\Cobar(C)$-algebra $\cV$. 
\end{defi}

Using the above link between $\Cobar(C)$-algebra structures on 
$\cV$ and Maurer-Cartan elements $Q$ of  $\coDer' \big(C(\cV) \big)$, we see that 
every homotopy algebra $\cV$ of type $C$ 
gives us a dg $C$-coalgebra $C(\cV)$ with the differential $\pa + Q$.
This observation motivates our definition of an $\infty$-morphism 
between homotopy algebras:
\begin{defi}
\label{dfn:infty-morph}
Let $\cA$, $\cB$ be homotopy algebras of type $C$ and let 
$Q_{\cA}$ (resp. $Q_{\cB}$) be the MC element of 
$\coDer' \big(C(\cA) \big)$ (resp. $ \coDer' \big(C(\cB) \big)$)
corresponding to the $\Cobar(C)$-algebra structure on $\cA$
(resp. $\cB$). Then an 
\emph{$\infty$-morphism} from $\cA$ to $\cB$ is a homomorphism 
$$
F : C(\cA) \to C(\cB)
$$
of the dg $C$-coalgebras $C(\cA)$ and $C(\cB)$ with the differentials 
$\pa + Q_{\cA}$ and $\pa + Q_{\cB}$, respectively. 

A homomorphism of dg $C$-coalgebras
$F$ is called an \emph{$\infty$ quasi-isomorphism} if the composition
\[
 \cA \hookrightarrow C(\cA) \stackrel{F}{\to} C(\cB) \stackrel{p_\cB}{\to} \cB
\]
is a quasi-isomorphism of cochain complexes.
\end{defi}

We say that two homotopy algebras $\cA$ and $\cB$ are 
\emph{quasi-isomorphic} if there exists a sequence of $\infty$ quasi-isomorphisms 
connecting $\cA$ with $\cB$.

\begin{defi}
\label{dfn:Def-comp}
Let $\cA$ be a homotopy algebra of type $C$ and 
$Q$ be the corresponding MC element of 
$\coDer' \big(C(\cA) \big)$. Then the cochain complex 
\begin{equation}
\label{Def-comp}
\Def(\cA) := \coDer' \big(C(\cA) \big) 
\end{equation}
with the differential $\pa + [Q, ~]$ is called the 
\emph{deformation complex} of the homotopy algebra $\cA$\,.
\end{defi}

\section{The main statement}
\label{sec:main}

We observe that the deformation complex  \eqref{Def-comp} of a homotopy 
algebra $\cA$ is naturally a dg Lie algebra.  We claim that 
\begin{thm}
\label{thm:main}
Let $C$ be a coaugmented dg cooperad satisfying Condition \ref{cond:C}.
If $\cA$ and $\cB$ are quasi-isomorphic homotopy algebras 
of type $C$ then the deformation complex 
$\Def(\cA)$ of $\cA$ is $L_{\infty}$-quasi-isomorphic to the deformation complex $\Def(\cB)$ of $\cB$.
\end{thm}
\begin{remark}
\label{rem:Keller}
For $A_{\infty}$-algebras this statement follows from the result
\cite{Keller} of B. Keller. 
\end{remark}

It is clearly sufficient to prove this theorem in the case when $\cA$ and $\cB$ are 
connected by a single $\infty$ quasi-isomorphism 
$F:\cA \leadsto \cB$. 

We will prove the Theorem by constructing an $L_\infty$-algebra $\Def(\cA \stackrel{F}{\leadsto} \cB)$, together with quasi-isomorphisms
\[
 \Def(\cA) ~\leftarrow~ \Def(\cA \stackrel{F}{\leadsto} \cB) ~\to~ \Def(\cB).
\]

The next subsections are concerned with the definition of $\Def(\cA \stackrel{F}{\leadsto} \cB)$.
The proof of Theorem \ref{thm:main} is given in Section \ref{sec:proof} below.

\subsection{The auxiliary $L_{\infty}$-algebra $\Cyl(C, \cA, \cB)$}

Let $\cA$, $\cB$ be cochain complexes. We consider the 
graded vector space 
\begin{equation}
\label{sprout}
\Cyl(C, \cA, \cB) :  = 
\Hom(C_{\c}(\cA), \cA) ~\oplus~
\bs \Hom(C(\cA), \cB) ~\oplus~
\Hom(C_{\c}(\cB), \cB)
\end{equation}
with the differential coming from those on $C$, $\cA$ and $\cB$\,. Here we denote by $\bs V$ the suspension of the graded vector space $V$. Concretely, if $v\in V$ has degree $d$, then the corresponding element $\bs v \in \bs V$ has degree $d+1$.

We equip the cochain complex $\Cyl(C, \cA, \cB)$ 
with an $L_{\infty}$-structure by declaring that 
\begin{equation}
\label{brack-cA-cA}
\{\bsi P_1, \bsi P_2, \dots, \bsi P_n\}_n : = 
\begin{cases}
 (-1)^{|P_1|+1} \bsi [P_1, P_2] & {\rm if} ~~ n =2\,, \\
 0 & {\rm otherwise}\,.
\end{cases}
\end{equation}
\begin{equation}
\label{brack-cB-cB}
\{\bsi R_1, \bsi R_2, \dots, \bsi R_n\}_n : = 
\begin{cases}
 (-1)^{|R_1|+1} \bsi [R_1, R_2] & {\rm if} ~~ n =2\,, \\
 0 & {\rm otherwise}\,.
\end{cases}
\end{equation}
for $P_i \in  \Hom(C_{\c}(\cA), \cA) \cong \Conv(C_{\c}, \End_{\cA})$\,,
and $R_i \in  \Hom(C_{\c}(\cB), \cB) \cong \Conv(C_{\c}, \End_{\cB})$\,,
and $[~,~]$ is the Lie bracket on the convolution algebras $\Conv(C_{\c}, \End_{\cA})$
and $\Conv(C_{\c}, \End_{\cB})$, respectively.

Furthermore,
\begin{multline}
\label{brack-cA-mixed}
\{ T, \bsi P\}_2 (X, a_1, \dots, a_n) = 
\\
\sum_{\substack{ 0 \le p \le n\\[0.1cm] \si \in \Sh_{p, n-p} }} \sum_{i}
(-1)^{|T| + |P| (|X'_{\si, i}| + 1)}
 T \big( X'_{\si, i}, P(X''_{\si, i}; a_{\si(1)}, \dots, a_{\si(p)}), a_{\si(p+1)}, \dots, a_{\si(n)}\big)\,,
\end{multline}
where $T \in  \Hom(C (\cA), \cB) $, $P \in   \Hom(C_{\c}(\cA), \cA) $, 
$X \in C_{\c}(n)$, $X'_{\si, i}$, $X''_{\si, i}$ are tensor factors in 
$$
\Delta_{\bt_{\si}} (X) = \sum_{i}  X'_{\si, i} \otimes  X''_{\si, i}\,,
$$
$P$ is extended by zero to $\cA \subset C(\cA)$, 
and $\bt_{\si}$ is the $n$-labeled planar tree depicted on figure  
\ref{fig:shuffle}.
\begin{figure}[htp]
\centering
\begin{tikzpicture}[scale=0.5]
\tikzstyle{w}=[circle, draw, minimum size=3, inner sep=1]
\tikzstyle{b}=[circle, draw, fill, minimum size=3, inner sep=1]
\node[b] (l1) at (-1, 4) {};
\draw (-1,4.6) node[anchor=center] {{\small $\si(1)$}};
\draw (0,3.8) node[anchor=center] {{\small $\dots$}};
\node[b] (lp) at (1, 4) {};
\draw (1,4.6) node[anchor=center] {{\small $\si(p)$}};
\node[w] (vv) at (0, 2) {};
\node[b] (lp1) at (3, 2.5) {};
\draw (3,3.1) node[anchor=center] {{\small $\si(p+1)$}};
\draw (4.25,2.4) node[anchor=center] {{\small $\dots$}};
\node[b] (ln) at (5.5, 2.5) {};
\draw (5.5,3.1) node[anchor=center] {{\small $\si(n)$}};
\node[w] (v) at (3, 1) {};
\node[b] (r) at (3, 0) {};
\draw (vv) edge (l1);
\draw (vv) edge (lp);
\draw (v) edge (vv);
\draw (v) edge (lp1);
\draw (v) edge (ln);
\draw (r) edge (v);
\end{tikzpicture}
\caption{\label{fig:shuffle} Here $\si$ is a $(p, n-p)$-shuffle}
\end{figure} 

To define yet another collection of non-zero $L_{\infty}$-brackets, we 
denote by $\Isom_{\pf}(m,r)$ the set of isomorphism classes of pitchforks
$\bt \in \Tree_{\pf}(m)$ with $r$ nodal vertices of height $2$. For every 
$z  \in \Isom_{\pf}(m,r)$ we choose a representative $\bt_{z}$ and denote by 
$X^k_{z, i}$ the tensor factors in 
\begin{equation}
\label{D-bt-z-X}
\Delta_{\bt_z}(X) = \sum_{i} X^0_{z, i} \otimes X^1_{z, i} \otimes \dots \otimes X^{r}_{z, i}\,,
\end{equation}
where $X \in C(m)$\,.

Finally, for vectors $T_j \in  \Hom(C(\cA), \cB)$ and $R \in  \Hom(C_{\c}(\cB), \cB)$ we set
\begin{multline}
\label{brack-cB-mixed}
\{\bsi R, T_1, \dots,  T_r \}_{r+1} (X, a_1, \dots, a_m) = 
\\
\sum_{\si \in S_{r}}
\sum_{z \in \Isom_{\pf}(m,r) }
\sum_i \pm (-1)^{|R|+1}
 R \big( X^0_{z, i}, T_{\si(1)}(X^1_{z, i}; a_{\la_z(1)}, \dots, a_{\la_z(n^z_1)}), 
\\
 T_{\si(2)}(X^2_{z, i}; a_{\la_z(n^z_1+1)}, \dots, a_{\la_z(n^z_1 + n^z_2)}), 
\dots ,  T_{\si(r)}(X^r_{z, i}; a_{\la_z(m-n^z_r+1)}, \dots, a_{\la_z(m)}) \big) \,,
\end{multline}
where $n^z_q$ is the number of leaves adjacent to the $(q+1)$-th nodal 
vertex of $\bt_z$, $\la_z(l)$ is the label of the $l$-th leaf of $\bt_z$, the 
map $R$ is extended by zero to $\cB \subset C(\cB)$ and the sign 
factor $\pm$ comes from the rearrangement of the homogeneous vectors  
\begin{equation}
\label{orig-order}
R, T_1, \dots,  T_r, X^0_{z, i}, X^1_{z, i}, \dots, X^r_{z, i}, a_1, \dots, a_m
\end{equation}
from their original positions in \eqref{orig-order} to their positions in  
the right hand side of \eqref{brack-cB-mixed}.

We observe that, due to axioms of a cooperad, the right hand side of 
\eqref{brack-cB-mixed} does not depend on the choice of representatives
$\bt_z \in  \Tree_{\pf}(m)$\,. 

The remaining $L_{\infty}$-brackets are either extended in the 
obvious way by symmetry or declared to be zero. 

We claim that 
\begin{claim}
\label{cl:L-infty-indeed}
The operations 
\begin{equation}
\label{brack-Cyl}
\{ ~,~,\dots,~ \}_n :  S^n (\bsi \Cyl(C, \cA, \cB))  \to  \bsi \Cyl(C, \cA, \cB), \qquad n \ge 2
\end{equation}
defined above have degree $1$ and satisfy the desired $L_{\infty}$-identities:
\begin{multline}
\label{Linf-iden}
\pa \{ f_1, f_2, \dots, f_n \}_n + 
\sum_{i=1}^n (-1)^{|f_1| + \dots  + |f_{i-1}| } \{ f_1, \dots,  f_{i-1}, \pa (f_i), f_{i+1}, \dots,  f_n \}_n  \\
\sum_{p=2}^{n-1} 
\sum_{\si \in \Sh_{p, n-p}} \pm 
\{ \{ f_{\si(1)}, f_{\si(2)}, \dots, f_{\si(p)} \}_p , f_{\si(p+1)}, \dots, f_{\si(n)} \}_{n-p+1}\,,
\end{multline}
where $f_j \in \bsi \Cyl(C, \cA, \cB)$ and the usual Koszul rule of signs is applied.
\end{claim}

Before proving Claim \ref{cl:L-infty-indeed},  we would like to 
show that
\begin{claim}
\label{cl:MC-Cyl}
The MC equation for the $L_{\infty}$-algebra  $\Cyl(C, \cA, \cB)$ 
is well defined. Moreover, MC elements of the $L_{\infty}$-algebra  $\Cyl(C, \cA, \cB)$ 
are triples: 
\begin{itemize}

\item a $\Cobar(C)$-algebra structure on $\cA$, 

\item a $\Cobar(C)$-algebra structure on $\cB$, and 

\item an $\infty$-morphism from $\cA$ to $\cB$\,. 

\end{itemize}

\end{claim}
\begin{proof}
\smartqed
Let $U$ be a degree $1$ element in $\Cyl(C, \cA, \cB)$.

We observe that the components of 
$$
\{\bsi U, \bsi U, \dots, \bsi U \}_n 
$$
in $\Hom(C_{\c}(\cA), \cA)$ and $\Hom(C_{\c}(\cB), \cB)$ are 
zero for all $n \ge 3$\,.
Furthermore, for every $(X, a_1, \dots, a_k) \in C(\cA)$ 
$$
\{\bsi U, \bsi U, \dots, \bsi U \}_n (X, a_1, \dots, a_k) =0 \qquad \forall~~ n \ge k+1\,. 
$$

Therefore the infinite sum 
\begin{equation}
\label{MC-U}
[\pa, U] + \sum_{n=2}^{\infty} \frac{1}{n!} \{\bsi U, \bsi U, \dots, \bsi U \}_n
\end{equation}
makes sense for every degree $1$ element $U$ in $\Cyl(C, \cA, \cB)$
and we can talk about MC elements of  $\Cyl(C, \cA, \cB)$.

To prove the second statement, we split the degree $1$ element $U \in \Cyl(C, \cA, \cB)$ into 
a sum 
$$
U = Q_{\cA} + \bs\, U_F + Q_{\cB}\,,
$$
where $Q_{\cA} \in \Conv(C_{\c}, \End_{\cA})$,  $Q_{\cB} \in \Conv(C_{\c}, \End_{\cB})$, 
and $U_F  \in \Hom(C(\cA), \cB)$\,.

Then the MC equation for $U$ is equivalent to the following three equations: 
\begin{equation}
\label{MC-Q-cA}
\pa Q_{\cA} + \frac{1}{2} [Q_{\cA}, Q_{\cA}] = 0 ~~\textrm{ in }~~ \Conv(C_{\c}, \End_{\cA})\,,
\end{equation}
\begin{equation}
\label{MC-Q-cB}
\pa Q_{\cB} + \frac{1}{2} [Q_{\cB}, Q_{\cB}] = 0 ~~\textrm{ in }~~    \Conv(C_{\c}, \End_{\cB})\,,
\end{equation}
and
\begin{equation}
\label{MC-mixed}
[\pa, U_F] + \{ U_F, \bsi Q_{\cA}\}_2  + 
\sum_{r=1}^{\infty} \frac{1}{r!} \{\bsi  Q_{\cB}, U_F, U_F, \dots, U_F\}_{r+1} = 0\,.
\end{equation}
 
Equations \eqref{MC-Q-cA} and \eqref{MC-Q-cB} imply that $Q_{\cA}$ (resp. $Q_{\cB}$)
gives us a $\Cobar(C)$-algebra structure on $\cA$ (resp. $\cB$)\,. Furthermore, equation 
\eqref{MC-mixed} means that $U_F$ is an $\infty$-morphism from $\cA$ to $\cB$\,.
\qed
\end{proof}

\subsubsection{Proof of Claim \ref{cl:L-infty-indeed}}

The most involved identity on $L_{\infty}$-brackets defined above is 
\begin{multline} 
\label{RRTTT}
\{\{\bsi R_1, \bsi R_2\}_2, T_1, \dots, T_n \}_{n+1} + 
\\
\sum_{\substack{ 1\le p \le n-1 \\[0.1cm] \si \in \Sh_{p, n-p}}  } \pm
\{\bsi R_1, \{\bsi R_2, T_{\si(1)}, \dots, T_{\si(p)} \}_{p+1},  T_{\si(p+1)}, 
\dots,  T_{\si(n)} \}_{n-p+2} +
\\
\sum_{\substack{ 1\le p \le n-1 \\[0.1cm] \si \in \Sh_{p, n-p}}  } \pm
\{\bsi R_2, \{\bsi R_1, T_{\si(1)}, \dots, T_{\si(p)} \}_{p+1},  T_{\si(p+1)}, 
\dots, T_{\si(n)} \}_{n-p+2} = 0\,.
\end{multline}

This identity is a consequence of a combinatorial 
fact about certain isomorphism classes in the groupoid $\Tree(n)$. 
To formulate this fact, we recall that the set of isomorphism classes
of $r$-labeled planar trees with two nodal vertices are in bijection with 
the set of shuffles
\begin{equation}
\label{shuffles}
\bigsqcup_{p=0}^r \Sh_{p,r-p}\,.
\end{equation}
This bijection assigns to a shuffle $\si \in  \Sh_{p, r-p}$ the $r$-labeled
planar tree $\bt_{\si}$ shown on figure \ref{fig:shuffle-r}. 
\begin{figure}[htp]
\centering
\begin{tikzpicture}[scale=0.5]
\tikzstyle{w}=[circle, draw, minimum size=3, inner sep=1]
\tikzstyle{b}=[circle, draw, fill, minimum size=3, inner sep=1]
\node[b] (l1) at (-1, 4) {};
\draw (-1,4.6) node[anchor=center] {{\small $\si(1)$}};
\draw (0,3.8) node[anchor=center] {{\small $\dots$}};
\node[b] (lp) at (1, 4) {};
\draw (1,4.6) node[anchor=center] {{\small $\si(p)$}};
\node[w] (vv) at (0, 2) {};
\node[b] (lp1) at (3, 2.5) {};
\draw (3,3.1) node[anchor=center] {{\small $\si(p+1)$}};
\draw (4.25,2.4) node[anchor=center] {{\small $\dots$}};
\node[b] (lr) at (5.5, 2.5) {};
\draw (5.5,3.1) node[anchor=center] {{\small $\si(r)$}};
\node[w] (v) at (3, 1) {};
\node[b] (r) at (3, 0) {};
\draw (vv) edge (l1);
\draw (vv) edge (lp);
\draw (v) edge (vv);
\draw (v) edge (lp1);
\draw (v) edge (lr);
\draw (r) edge (v);
\end{tikzpicture}
\caption{\label{fig:shuffle-r} Here $\si$ is a $(p, r-p)$-shuffle}
\end{figure}
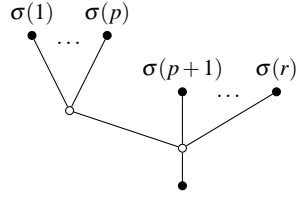

Next, we observe that $\pi_0 (\Tree_{\pf}(n))$  
is in bijection with the set 
\begin{equation}
\label{Isom-pf-desc}
\bigsqcup_{r \ge 1} \mSH_{n,r},
\end{equation}
where\footnote{It is obvious that, for every $\tau \in \mSH_{n,r}$, $\tau(1)=1$\,.}
\begin{equation}
\label{mSH-n-r}
\mSH_{n,r} = 
\end{equation}
$$
 \bigsqcup_{1 \le q_1 < q_2 < \dots < q_{r-1} < q_r = n}  
\{\tau \in \Sh_{q_1, q_2-q_1, \dots, n-q_{r-1}} ~|~ \tau(1)<  \tau(q_1+1) <   
\tau(q_2+1) < \dots < \tau(q_{r-1}+1)\}\,.
$$
This bijection assigns to a shuffle $\tau$ in the set \eqref{Isom-pf-desc}
the isomorphism class of the pitchfork $\bt^{\pf}_{\tau}$ depicted on figure \ref{fig:pf-tau}.
\begin{figure}[htp] 
\centering 
\begin{tikzpicture}[scale=1]
\tikzstyle{w}=[circle, draw, minimum size=3, inner sep=1]
\tikzstyle{b}=[circle, draw, fill, minimum size=3, inner sep=1]
\node[b] (l1) at (0, 3) {};
\draw (0,3.3) node[anchor=center] {{\small $\tau(1)$}};
\draw (0.5,2.8) node[anchor=center] {{\small $\dots$}};
\node[b] (lq1) at (1, 3) {};
\draw (1,3.3) node[anchor=center] {{\small $\tau(q_1)$}};
\node[b] (lq11) at (2.5, 3) {};
\draw (2.5,3.3) node[anchor=center] {{\small $\tau(q_1+1)$}};
\draw (3.5,2.8) node[anchor=center] {{\small $\dots$}};
\node[b] (lq2) at (4.5, 3) {};
\draw (4.5,3.3) node[anchor=center] {{\small $\tau(q_2)$}};
\node[b] (lahz) at (7, 3) {};
\draw (7,3.3) node[anchor=center] {{\small $\tau(q_{r-1})$}};
\draw (8,2.8) node[anchor=center] {{\small $\dots$}};
\node[b] (ln) at (9, 3) {};
\draw (9,3.3) node[anchor=center] {{\small $\tau(n)$}};
\node[w] (v1) at (0.5, 2) {};
\node[w] (v2) at (3.5, 2) {};
\draw (5.75,2) node[anchor=center] {{\large $\dots$}};
\node[w] (vr) at (8, 2) {};
\node[w] (v) at (4.5, 1) {};
\node[b] (r) at (4.5, 0.5) {};
\draw (r) edge (v);
\draw (v) edge (v1);
\draw (v) edge (v2);
\draw (v) edge (vr);
\draw (v1) edge (l1);
\draw (v1) edge (lq1);
\draw (v2) edge (lq11);
\draw (v2) edge (lq2);
\draw (vr) edge (lahz);
\draw (vr) edge (ln);
\end{tikzpicture}
\caption{The pitchfork $\bt^{\pf}_{\tau}$} \label{fig:pf-tau}
\end{figure}

Note that, in the degenerate cases $r=1$ and $r=n$, 
$\mSH_{n,r}$ is the one-element set consisting of the identity 
permutation $\id \in S_n$\,. 
The corresponding pitchforks are shown on figures 
\ref{fig:pf-degen-1} and \ref{fig:pf-degen-n}, respectively. 
\begin{figure}[htp] 
\centering 
\begin{minipage}[t]{0.45\linewidth}
\centering 
\begin{tikzpicture}[scale=1]
\tikzstyle{w}=[circle, draw, minimum size=3, inner sep=1]
\tikzstyle{b}=[circle, draw, fill, minimum size=3, inner sep=1]
\node[b] (l1) at (0, 3) {};
\draw (0,3.3) node[anchor=center] {{\small $1$}};
\node[b] (l2) at (1, 3) {};
\draw (1,3.3) node[anchor=center] {{\small $2$}};
\draw (2,3) node[anchor=center] {{\small $\dots$}};
\node[b] (ln) at (3, 3) {};
\draw (3,3.3) node[anchor=center] {{\small $n$}};
\node[w] (v1) at (1.5, 2) {};
\node[w] (v0) at (1.5, 1) {};
\node[b] (r) at (1.5, 0) {};
\draw (r) edge (v0);
\draw (v0) edge (v1);
\draw (v1) edge (l1);
\draw (v1) edge (l2);
\draw (v1) edge (ln);
\end{tikzpicture}
\caption{The pitchfork for $r=1$} \label{fig:pf-degen-1}
\end{minipage}
\begin{minipage}[t]{0.45\linewidth}
\centering 
\begin{tikzpicture}[scale=1]
\tikzstyle{w}=[circle, draw, minimum size=3, inner sep=1]
\tikzstyle{b}=[circle, draw, fill, minimum size=3, inner sep=1]
\node[b] (l1) at (0, 3) {};
\draw (0,3.3) node[anchor=center] {{\small $1$}};
\node[b] (l2) at (1, 3) {};
\draw (1,3.3) node[anchor=center] {{\small $2$}};
\draw (2,3) node[anchor=center] {{\small $\dots$}};
\node[b] (ln) at (3, 3) {};
\draw (3,3.3) node[anchor=center] {{\small $n$}};
\node[w] (v1) at (0, 2) {};
\node[w] (v2) at (1, 2) {};
\node[w] (vn) at (3, 2) {};
\node[w] (v0) at (1.5, 1) {};
\node[b] (r) at (1.5, 0) {};
\draw (r) edge (v0);
\draw (v0) edge (v1);
\draw (v0) edge (v2);
\draw (v0) edge (vn);
\draw (v1) edge (l1);
\draw (v2) edge (l2);
\draw (vn) edge (ln);
\end{tikzpicture}
\caption{The pitchfork for $r=n$} \label{fig:pf-degen-n}
\end{minipage}
\end{figure}

For every permutation $\tau \in \mSH_{n,r}$ and a shuffle $\si \in \Sh_{p, r-p}$
we can form the following $n$-labeled planar tree
\begin{equation}
\label{pf-tau-sigma}
\bt^{\pf}_{\tau} \bullet_1 \bt_{\si}\,, 
\end{equation}
where $\bt \bullet_j \bt'$ denotes the insertion of the tree $\bt'$ into 
the $j$-th nodal vertex of the tree $\bt$ (see Section 2.2 in \cite{notes}). 

It is clear that, for distinct pairs $(\tau, \si) \in  \mSH_{n,r} \times \Sh_{p, r-p}$,
we get mutually non-isomorphic labeled planar trees.

Let $\tau' \in  \mSH_{n,r'}$ and
$m_i$ be the number of edges which terminate at the $(i+1)$-th 
nodal vertex of $\bt^{\pf}_{\tau'}$. 
For every $\tau'' \in  \mSH_{m_i,r''}$, we may form the $n$-labeled 
planar tree 
\begin{equation}
\label{pf-pf}
\bt^{\pf}_{\tau'} \bullet_{i+1} \bt^{\pf}_{\tau''}\,.
\end{equation}

It is clear that, for distinct triples 
$(\tau', i, \tau'') \in  \mSH_{n,r'} \times \{1,2, \dots, r'\} \times   \mSH_{m_i,r''}$,
the corresponding labeled planar trees \eqref{pf-pf} are mutually non-isomorphic. 
Furthermore, every tree of the form \eqref{pf-pf} is isomorphic to exactly one 
tree of the form \eqref{pf-tau-sigma} and vice versa. 
This is precisely the combinatorial fact that is need to prove that identity
\eqref{RRTTT} holds. 

Indeed, the terms in the expression 
$$
\{\{\bsi R_1, \bsi R_2\}_2, T_1, \dots, T_n \}_{n+1}
$$
involve trees of the form \eqref{pf-tau-sigma} and the
terms in the expressions 
$$
\sum_{\substack{ 1\le p \le n-1 \\[0.1cm] \si \in \Sh_{p, n-p}}  } \pm
\{\bsi R_1, \{\bsi R_2, T_{\si(1)}, \dots, T_{\si(p)} \}_{p+1},  T_{\si(p+1)}, 
\dots,  T_{\si(n)} \}_{n-p+2}
$$
and
$$
\sum_{\substack{ 1\le p \le n-1 \\[0.1cm] \si \in \Sh_{p, n-p}}  } \pm
\{\bsi R_2, \{\bsi R_1, T_{\si(1)}, \dots, T_{\si(p)} \}_{p+1},  T_{\si(p+1)}, 
\dots, T_{\si(n)} \}_{n-p+2}
$$
involve trees of the form \eqref{pf-pf}.

Thus it only remains to check that the sign factors match. 

The remaining identities on $L_{\infty}$-brackets are simpler
and we leave their verification to the reader.  

Claim \ref{cl:L-infty-indeed} is proved. \smartqed\qed

\subsection{The $L_{\infty}$-algebra $\Cyl(C, \cA, \cB)^{\bs F_1}$ and 
its MC elements}

Let 
\begin{equation}
\label{F-1}
F_1 : \cA  \to \cB
\end{equation}
be a map of cochain complexes.

We may view $\bs F_1$ as a degree $1$ element in $\Cyl(C, \cA, \cB)$: 
$$
\bs F_1 \in \bs \Hom (\cA, \cB) \subset \bs\Hom(C(\cA), \cB) \subset \Cyl(C, \cA, \cB)\,.
$$
Since $F_1$ is compatible with the differentials on $\cA$ and $\cB$, 
$\bs F_1$ is obviously a MC element of $\Cyl(C, \cA, \cB)$ and, 
in view of Claim \ref{cl:MC-Cyl}, $\bs F_1$ corresponds to the 
triple: 
\begin{itemize}

\item the trivial $\Cobar(C)$-algebra structure on $\cA$,

\item the trivial $\Cobar(C)$-algebra structure on $\cB$, and 

\item a strict\footnote{i.e. an $\infty$-morphism $F: \cA \leadsto \cB$ whose all 
higher structure maps are zero.} $\infty$-morphism $F_1$ from $\cA$ to $\cB$. 

\end{itemize}

Let  $Q_1, Q_2, \dots, Q_m $ be vectors in $ \Cyl(C, \cA, \cB)$. We recall that 
the components 
of 
$$
\{\underbrace{F_1,  F_1, \dots, F_1}_{n \textrm{ times}}, \bsi Q_1, \bsi Q_2, \dots, \bsi Q_m \}_{n+m}
$$
in $\Hom(C_{\c}(\cA), \cA)$ and  $\Hom(C_{\c}(\cB), \cB)$ are zero if $n+m > 2$\,.
Furthermore, for every  $(X, a_1, \dots, a_k) \in C(\cA)$ 
$$
\{\underbrace{F_1,  F_1, \dots, F_1}_{n \textrm{ times}}, \bsi Q_1, \bsi Q_2, \dots, \bsi Q_m \}_{n+m}
(X, a_1, \dots, a_k) = 0 ~ (\in \cB)
$$
provided $n+m \ge k+2$\,.

Therefore we may twist (see \cite[Remark 3.11.]{DeligneTw}) the $L_{\infty}$-algebra 
on $\Cyl(C, \cA, \cB)$ by the MC element $\bs F_1$. 
We denote by 
\begin{equation}
\label{Cyl-AB-F1}
\Cyl(C, \cA, \cB)^{\bs F_1}
\end{equation}
the $L_{\infty}$-algebra obtained  in this way. 

It is not hard to see that\footnote{In $\Cyl_{\c}(C, \cA, \cB)^{\bs F_1}$, we have  $\bs\Hom(C_{\c}(\cA), \cB)$ 
instead of $ \bs\Hom(C(\cA), \cB)$\,.}  
\begin{equation}
\label{Cyl-c-AB-F1}
\Cyl_{\c}(C, \cA, \cB)^{\bs F_1} :=  \Hom(C_{\c}(\cA), \cA) \oplus
 \bs\Hom(C_{\c}(\cA), \cB) \oplus  \Hom(C_{\c}(\cB), \cB)
\end{equation}
is an $L_{\infty}$-subalgebra of $\Cyl(C, \cA, \cB)^{\bs F_1}$\,. Furthermore, 
Claim \ref{cl:MC-Cyl} implies that 
\begin{claim}
\label{cl:MC-Cyl-F1}
MC elements of the $L_{\infty}$-algebra \eqref{Cyl-c-AB-F1} are 
triples: 
\begin{itemize}

\item  A $\Cobar(C)$-algebra structure on $\cA$, 

\item  A $\Cobar(C)$-algebra structure on $\cB$,

\item an $\infty$-morphism $F : \cA \leadsto \cB$ for which 
the composition
$$
\cA \hookrightarrow C(\cA) \stackrel{F}{ \longrightarrow } C(\cB)
 \stackrel{p_{\cB}}{ \longrightarrow }  \cB  
$$ 
coincides with $F_1$\,.
\end{itemize}
\smartqed\qed
\end{claim}
\begin{remark}
\label{rem:filtration-Cyl-c}
Using the ascending filtration \eqref{C-circ-filtr} on the 
pseudo-operad $C_{\c}$, we equip the $L_{\infty}$-algebra 
$\Cyl(C, \cA, \cB)$ with the complete descending filtrations: 
\begin{equation}
\label{filtr-Cyl}
\Cyl(C, \cA, \cB) = \cF_0  \Cyl(C, \cA, \cB) \supset  \cF_1  \Cyl(C, \cA, \cB)
\supset  \cF_2  \Cyl(C, \cA, \cB)  \supset \dots\,, 
\end{equation}
where (for $m \ge 1$)
$$
\cF_m  \Cyl(C, \cA, \cB) : = 
$$
\begin{equation}
\label{filtr-Cyl-dfn}
\big\{  Q' \oplus F \oplus Q'' \in  \Hom(C_{\c}(\cA), \cA) \oplus
\bs \Hom(C(\cA), \cB) \oplus  \Hom(C_{\c}(\cB), \cB) 
\end{equation}
$$
Q' (X, a_1, \dots, a_k) = 0\,, \quad 
F(X, a_1, \dots, a_k) = 0\,, \quad 
Q'' (X, b_1, \dots, b_k) =0 \quad \forall ~
X \in \cF^{m-1} C(k) \big\}\,.
$$
The same formulas define a complete descending filtration on 
the $L_{\infty}$-algebras $\Cyl(C, \cA, \cB)^{\bs F_1}$ and 
$\Cyl_{\c}(C, \cA, \cB)^{\bs F_1}$\,. 

We observe that 
\begin{equation}
\label{Cyl-c-advantage}
\Cyl_{\c}(C, \cA, \cB)^{\bs F_1} = \cF_1 \Cyl_{\c}(C, \cA, \cB)^{\bs F_1} 
\end{equation}
and hence $\Cyl_{\c}(C, \cA, \cB)^{\bs F_1} $ is pro-nilpotent. 
Later, we will use this advantage of  $\Cyl_{\c}(C, \cA, \cB)^{\bs F_1}$
over $\Cyl(C, \cA, \cB)^{\bs F_1}$\,.
\end{remark}

\subsection{What if $F_1$ is a quasi-isomorphism?}

Starting with a chain map \eqref{F-1} we define two maps
of cochain complexes: 
\begin{equation}
\label{f}
P \mapsto f(P) = F_1 \circ P  ~ : ~
\Hom(C_{\c}(\cA), \cA) \to \Hom(C_{\c}(\cA), \cB)
\end{equation}
\begin{equation}
\label{wt-f}
R \mapsto \wt{f}(R) = R \circ C_{\c}(F_1)   ~ : ~
\Hom(C_{\c}(\cB), \cB) \to \Hom(C_{\c}(\cA), \cB)
\end{equation}
and observe that the cochain complex $\Cyl_{\c}(C, \cA, \cB)^{\bs F_1} $ is precisely
the cochain complex $\Cyl(f, \wt{f})$ defined in \eqref{Cyl-f-wtf}, \eqref{diff-Cyl} 
in the Appendix.

Hence, using Lemma \ref{lem:pi-q-iso}, we deduce the following statement: 
\begin{prop}
\label{prop:proj-s-are-q-iso}
If the chain map $F_1 :\cA \to \cB $ induces an isomorphism on 
the level of cohomology then so do the following canonical projections: 
\begin{equation}
\label{pi-cA}
\pi_{\cA} : \Cyl_{\c}(C, \cA, \cB)^{\bs F_1}
\to  \Hom(C_{\c}(\cA), \cA)\,,
\end{equation}
\begin{equation}
\label{pi-cB}
\pi_{\cB} : \Cyl_{\c}(C, \cA, \cB)^{\bs F_1}
 \to  \Hom(C_{\c}(\cB), \cB)\,.
\end{equation}
The maps $\pi_{\cA}$ and $\pi_{\cB}$ are strict homomorphisms of 
$L_{\infty}$-algebras.  
\end{prop}
\begin{proof}
\smartqed
Since we work over a field of characteristic zero, the functors 
$\Hom$, $\otimes$, as well as the functors of taking (co)invariants
with respect to actions of symmetric groups preserve quasi-isomorphisms. 
Therefore the maps \eqref{f} and \eqref{wt-f} are quasi-isomorphisms 
of cochain complexes. 

Thus the first statement follows directly from  Lemma \ref{lem:pi-q-iso}.

The second statement is an obvious consequence of the definition of
$L_{\infty}$-brackets on $ \Cyl_{\c}(C, \cA, \cB)^{\bs F_1}$\,. 
\qed
\end{proof}

\subsection{Proof of Theorem \ref{thm:main}} 
\label{sec:proof}

We will now give a proof of Theorem \ref{thm:main}

Let $\cA$ and $\cB$ be homotopy algebras of type $C$\,. 
As said above, we may assume, without loss of generality, that $\cA$ and $\cB$ 
are connected by a single $\infty$ quasi-isomorphism: 
\begin{equation}
\label{F-cA-cB}
F : \cA \leadsto \cB\,.
\end{equation}

We denote by $\al^{\Cyl}$ the MC element of  $\Cyl_{\c}(C, \cA, \cB)^{\bs F_1}$ which 
corresponds to the triple
\begin{itemize}
\item the homotopy algebra structure on $\cA$, 
\item the homotopy algebra structure on $\cB$, and
\item the $\infty$-morphism $F$\,.
\end{itemize}

Due to \eqref{Cyl-c-advantage}, we may twist  (see \cite[Remark 3.11.]{DeligneTw}) the 
$L_{\infty}$-algebra   $\Cyl_{\c}(C, \cA, \cB)^{\bs F_1}$ by 
the MC element $\al^{\Cyl}$\,. We denote by 
\begin{equation}
\label{Def-cA-cB}
\Def(\cA \stackrel{F}{\leadsto} \cB)
\end{equation}
the $L_{\infty}$-algebra which is obtained from $\Cyl_{\c}(C, \cA, \cB)^{\bs F_1}$ 
via twisting by the MC element $\al^{\Cyl}$. 

We also denote by $Q_{\cA}$ (resp. $Q_{\cB}$) the MC element of 
$\Conv(C_{\c}, \End_{\cA})$ (resp. $\Conv(C_{\c}, \End_{\cB})$) corresponding 
to the homotopy algebra structure on $\cA$ (resp. $\cB$) and recall 
that $\Def(\cA)$ (resp. $\Def(\cB)$) is obtained from  $\Conv(C_{\c}, \End_{\cA})$ 
(resp. $\Conv(C_{\c}, \End_{\cB})$) via twisting by the 
MC element $Q_{\cA}$ (resp. $Q_{\cB}$). 

It is easy to see that 
\begin{equation}
\label{pi-MC}
\pi_{\cA}(\al^{\Cyl}) = Q_{\cA}\,, 
\qquad 
\pi_{\cB}(\al^{\Cyl}) = Q_{\cB}\,. 
\end{equation}

Since $\pi_{\cA}$ \eqref{pi-cA} and $\pi_{\cB}$ \eqref{pi-cB} are strict 
$L_{\infty}$-morphisms, they do not change under twisting by MC elements. 
Thus, we conclude that, the same maps $\pi_{\cA}$ and $\pi_{\cB}$ 
give us (strict) $L_{\infty}$-morphisms 
\begin{equation}
\label{pi-cA-pi-cB}
\begin{array}{c}
\pi_{\cA} : \Def(\cA \stackrel{F}{\leadsto} \cB)  \to \Def(\cA)\,,  \\[0.3cm]
\pi_{\cB} :  \Def(\cA \stackrel{F}{\leadsto} \cB) \to \Def(\cB)\,. 
\end{array}
\end{equation}

According to \cite[Proposition 6.2]{DeligneTw}, twisting preserves quasi-isomorphisms. 
Thus, due to Proposition \ref{prop:proj-s-are-q-iso}, the two arrows in \eqref{pi-cA-pi-cB} are 
(strict) $L_{\infty}$-quasi-isomorphisms, as desired. 

Theorem \ref{thm:main} is proven. \smartqed\qed

\section{Sheaves of homotopy algebras}
\label{sec:sheaves}

For a topological space $X$ we consider the category $\dgSh_X$ of dg sheaves
(i.e. sheaves of unbounded cochain complexes of $\bbK$-vector spaces). 
We recall that $\dgSh_X$ is a symmetric monoidal category for which the monoidal product is 
the tensor product followed by sheafification. 

Given coaugmented dg cooperad $C$ (satisfying condition \eqref{cond:C}) 
one may give the following naive definition of a homotopy algebra of type 
$C$ in the category $\dgSh_X$:  
\begin{defi}[Naive!]
\label{dfn:naive-alg}
We say that a dg sheaf $\cA$ on $X$ carries a structure of \emph{a homotopy 
algebra of type} $C$ if $\cA$ is an algebra over the dg operad $\Cobar(C)$\,.
\end{defi}
One can equivalently define  a homotopy algebra of type $C$ by 
considering coderivations of 
the cofree $C$-coalgebra (in the category $\dgSh_X$) 
\begin{equation}
\label{C-cA}
C(\cA) := \bigoplus_n \left( C(n)\otimes \cA^{\otimes n} \right)_{S_n}.
\end{equation}
In other words, a homotopy algebra of type $C$ on $\cA$ is a degree $1$ 
coderivation $\cQ$ of $C(\cA)$ satisfying the MC equation and the additional 
condition
$$
\cQ \Big|_{\cA} ~ = ~0\,.
$$

Given such a coderivation $\cQ$, it is natural to consider the $C$-coalgebra
\eqref{C-cA} with the new differential 
\begin{equation}
\label{diff-C-cA}
\pa + \cQ
\end{equation}
where $\pa$ comes from the differentials on $C$ and $\cA$. 

This observation motivates the following naive definition of $\infty$-morphism 
of homotopy algebra in $\dgSh_X$: 
\begin{defi}[Naive!]
\label{dfn:naive-morph}
Let $\cA$ and $\cB$ be homotopy algebras of type $C$ in 
$\dgSh_X$ and let $\cQ_{\cA}$ and $\cQ_{\cB}$ be the 
corresponding coderivations of $C(\cA)$ and $C(\cB)$ respectively. 
An $\infty$-\emph{morphism} $F : \cA \leadsto \cB$ is a map of 
sheaves 
$$
F : C(\cA) \to C(\cB) 
$$  
which is compatible with the $C$-coalgebra structure and 
the differentials $\pa+ \cQ_{\cA}$, $\pa + \cQ_{\cB}$\,. 
\end{defi}

An important disadvantage of the above naive definitions is that they 
do not admit an analogue of the homotopy transfer theorem 
\cite[Theorem 10.3.2]{LV-book}.
For this reason we propose ``more mature'' definitions based on 
the use of the Thom-Sullivan normalization \cite{BG}, \cite[Appendix A]{VdB}.
 
Let $\mU$ be a covering of $X$ and $\cA$ be a dg sheaf on $X$\,.  
The associated cosimplicial set $\mU(\cA)$ is naturally a cosimplicial 
cochain complex. So, applying the Thom-Sullivan functor $N^{\TS}$ to 
$\mU(\cA)$, we get a cochain complex 
\begin{equation}
\label{N-TS-mU-cA}
N^{\TS} \mU(\cA)
\end{equation}
which computes the Cech hyper-cohomology of $\cA$ with respect to 
the cover $\mU$. 

Let us assume, for simplicity, that there exists an acyclic covering $\mU$ for $\cA$. 
In particular, $\check H_\mU(\cA)\cong H(\cA)$ agrees with the sheaf cohomology of $\cA$. 

Then, we have the following definition:
\begin{defi}
\label{dfn:homot-alg-Sh}
\emph{A homotopy algebra structure of type} $C$ on a dg 
sheaf $\cA$ is $\Cobar(C)$-algebra structure on 
the cochain complex \eqref{N-TS-mU-cA}. 
\end{defi}
\begin{remark}
\label{rem:from-naive-alg}
Since, the Thom-Sullivan normalization $N^{\TS}$ is a symmetric monoidal functor 
from cosimplicial cochain complexes into cochain complexes, a homotopy algebra 
structure on $\cA$ in the sense of naive Definition \ref{dfn:naive-alg} is a homotopy 
algebra structure on $\cA$ in the sense of Definition \ref{dfn:homot-alg-Sh}. 
\end{remark}
\begin{remark}
\label{rem:cover}
Let $\mU'$ be another acyclic covering of $X$ and $\mV$ be a common acyclic refinement 
of $\mU$ and $\mU'$. Since the functor $N^{\TS}$ preserves quasi-isomorphisms,
the cochain complexes $N^{\TS} \mU(\cA)$ and $N^{\TS} \mU'(\cA)$ are connected 
by the following pair of quasi-isomorphisms:
\begin{equation}
\label{pair-q-iso-N-TS}
N^{\TS} \mU(\cA) \stackrel{\sim}{\,\longrightarrow\,}  N^{\TS}\mV(\cA) \stackrel{\sim}{\, \longleftarrow\,} 
N^{\TS}\mU'(\cA)\,.
\end{equation}
Hence, using the usual homotopy transfer theorem \cite[Theorem 10.3.2]{LV-book}, 
we conclude that the notion of homotopy algebra structure on a dg sheaf $\cA$
is, in some sense, independent on the choice of acyclic covering.
\end{remark}

Proceeding further in this fashion, we give the definition of 
an $\infty$-morphism (and $\infty$ quasi-isomorphism) in the setting of sheaves:
\begin{defi}
\label{dfn:morph-Sh}
Let $\cA$ and $\cB$ be dg sheaves on $X$ equipped with structures 
of homotopy algebras of type $C$. \emph{An $\infty$-morphism} $F$ from 
$\cA$ to $\cB$ is an $\infty$-morphism 
\begin{equation}
\label{F-N-TS}
F : N^{\TS} \mU(\cA) \leadsto N^{\TS} \mU(\cB)
\end{equation}
of the corresponding homotopy algebras (in the category of cochain complexes)
for some acyclic cover $\mU$. If \eqref{F-N-TS} is an $\infty$ quasi-isomorphism 
then, we say that,  $F$ is an $\infty$ \emph{quasi-isomorphism} from $\cA$ to $\cB$\,.
\end{defi}
\begin{remark}
\label{rem:from-naive-morph}
Again, since the Thom-Sullivan normalization $N^{\TS}$ is a symmetric monoidal functor 
from cosimplicial cochain complexes into cochain complexes, an $\infty$-morphism 
in the sense of naive Definition \ref{dfn:naive-morph} gives us an $\infty$-morphism 
in the of Definition \ref{dfn:morph-Sh}. 
\end{remark}

\subsection{The deformation complex in the setting of sheaves}
Let $X$ be a topological space and $\cA$ be a dg sheaf on $X$. 
Let us assume that $\mU$ is an acyclic (for $\cA$) cover of $X$
and $\cA$ carries a homotopy algebra of type $C$ defined in terms 
of this cover $\mU$. 

\begin{defi}
\label{dfn:Def-Sh}
\emph{The deformation complex} of the sheaf of homotopy 
algebras $\cA$ is 
\[
\Def(\cA):= \Def(N^{\TS} \mU(\cA)).
\]
\end{defi}
\begin{remark}
\label{rem:indep-cover}
The above definition of the deformation complex is
independent on the choice of the acyclic cover in the following 
sense: Let $\mU'$ be another acyclic cover of $X$. Since the 
cochain complexes $N^{\TS} \mU(\cA)$ and $N^{\TS} \mU'(\cA)$
are connected by the pair of quasi-isomorphisms  
\eqref{pair-q-iso-N-TS}, Theorem \ref{thm:main} and the homotopy 
transfer theorem imply that the deformation complexes corresponding
to different acyclic coverings are connected by a sequence of quasi-isomorphisms 
of dg Lie algebras.
\end{remark}

Theorem \ref{thm:main} has the following obvious implication 
\begin{cor}
\label{cor:main}
Let $\cA$ and $\cB$ be dg sheaves on $X$ equipped with structures 
of homotopy algebras of type $C$. If $\cA$ and $\cB$ are connected by 
a sequence of $\infty$ quasi-isomorphisms then $\Def(\cA)$ and $\Def(\cB)$
are quasi-isomorphic dg Lie algebras. \smartqed\qed
\end{cor}

\subsection{An application of Corollary \ref{cor:main}}

In applications we often deal with honest (versus $\infty$) algebraic 
structures on sheaves and maps of sheaves which are 
compatible with these algebraic structures on the nose 
(not up to homotopy). Here we describe a setting of this kind 
in which Corollary \ref{cor:main} can be applied. 

Let $O$ be a dg operad and $\Cobar(C)$ be a resolution of $O$
for which the cooperad $C$ satisfies condition \eqref{cond:C}.  

Every dg sheaf of $O$-algebras $\cA$ is naturally a sheaf of 
$\Cobar(C)$-algebras. Hence, $\cA$ carries a structure of homotopy 
algebra of type $C$ and we define the deformation complex of $\cA$ as
$$
\Def(\cA) : = \Def(N^{\TS} \mU(\cA))\,.
$$

\begin{thm}
\label{thm:applic}
Let $\cA$ and $\cB$ be dg sheaves of $O$-algebras on 
a topological space $X$. If there exists a sequence of quasi-isomorphisms 
of dg sheaves of $O$-algebras
$$
\cA \stackrel{\sim}{\,\leftarrow\, } \cA_1 \stackrel{\sim}{\,\rightarrow\, }
\cA_2  \stackrel{\sim}{\,\leftarrow\, } \dots   \stackrel{\sim}{\,\rightarrow\, } \cA_n
 \stackrel{\sim}{\,\rightarrow\, } \cB
$$
then the dg Lie algebras $\Def(\cA)$ and $\Def(\cB)$ are quasi-isomorphic.
\end{thm}
\begin{proof}
\smartqed
It is suffices to prove this theorem in the case when $\cA$ and $\cB$
are connected by a single quasi-isomorphism 
\begin{equation}
\label{f-cA-cB}
f :  \cA \stackrel{\sim}{\,\rightarrow\, }  \cB
\end{equation}
of dg sheaves of $O$-algebras. 

Since the functor $N^{\TS}$ preserves quasi-isomorphisms, 
$f$ induces a quasi-isomorphism 
\begin{equation}
\label{f-star}
f_* : N^{\TS}\mU(\cA) \stackrel{\sim}{\,\rightarrow\, }  N^{\TS}\mU(\cB)
\end{equation}
for any acyclic cover $\mU$. 

Furthermore, since $N^{\TS}$ is compatible with the symmetric monoidal 
structure, the map $f_*$ is compatible with the $O$-algebra 
structures on  $N^{\TS}\mU(\cA)$ and $ N^{\TS}\mU(\cB)$. 

Therefore, $f_*$ may be viewed as an $\infty$ quasi-isomorphism 
from $\cA$ to $\cB$.   

Thus  Corollary \ref{cor:main} implies the desired statement.
\qed 
\end{proof}

\subsection{A Concluding remark about Definitions \ref{dfn:homot-alg-Sh}, 
\ref{dfn:morph-Sh}, and \ref{dfn:Def-Sh}} 

For certain applications, Definitions \ref{dfn:homot-alg-Sh}, 
\ref{dfn:morph-Sh}, and \ref{dfn:Def-Sh} may still be naive. 
One may ask about a possibility to extend the notion of homotopy 
algebras to the setting of twisted complexes \cite{OB}, \cite{OTT}, 
\cite{Chern-Twisted},  \cite{Gillet}. For some application one may need a universal way 
of keeping track on ``dependencies on covers'' by using the notion 
of hypercover. For other applications one may need a notion of 
deformation complex which would also govern deformations of 
$\cA$ as a sheaf or possibly as a (higher) stack. 

However, for applications considered in \cite{Chern}, 
the framework of  Definitions \ref{dfn:homot-alg-Sh}, 
\ref{dfn:morph-Sh}, and \ref{dfn:Def-Sh} is sufficient.


\section*{Appendix: Cylinder type construction}
\addcontentsline{toc}{section}{Appendix: Cylinder type construction}

Given a pair $(f, \wt{f})$ of maps of cochain complexes
\begin{equation}
\label{pair}
V \stackrel{f}{\longrightarrow} W  \stackrel{\wt{f}}{\longleftarrow} \wt{V}\,,
\end{equation}
we form another cochain complex $\Cyl(f, \wt{f})$\,.
As a graded vector space 
\begin{equation}
\label{Cyl-f-wtf}
\Cyl(f, \wt{f}) := V \oplus \bs W \oplus \wt{V}
\end{equation}
and the differential $\pa^{\Cyl}$ is defined by the formula: 
\begin{equation}
\label{diff-Cyl}
\pa^{\Cyl} (v + \bs w + \wt{v})  := 
\pa v + \bs (f(v) - \pa w + \wt{f}(\wt{v} )) + \pa \wt{v}\,.    
\end{equation}
The equation 
$$
\pa^{\Cyl}  \circ \pa^{\Cyl} = 0
$$
is a consequence of $\pa^2 = 0$ and the compatibility of 
$f$ (resp. $\wt{f}$) with the differentials\footnote{By abuse of notation, we denote 
by the same letter $\pa$ the differential on $V$, $W$, and $\wt{V}$.}
on $V$, $W$, and $\wt{V}$\,. 

We have the obvious pair of maps of cochain complexes: 
\begin{equation}
\label{projections}
V \stackrel{\pi_{V}}{\longleftarrow} 
\Cyl(f, \wt{f}) \stackrel{\pi_{\wt{V}}}{\longrightarrow} \wt{V}
\end{equation}
\begin{equation}
\label{pi-V}
\pi_{V} (v + \bs w + \wt{v}) = v, 
\end{equation}
\begin{equation}
\label{pi-wtV}
\pi_{\wt{V}} (v + \bs w + \wt{v}) = \wt{v}.  
\end{equation}
  
We claim that 
\begin{lemma}
\label{lem:pi-q-iso}
If $f$ and $\wt{f}$ are quasi-isomorphisms of cochain complexes, 
then so are $\pi_{\cV}$ and $\pi_{\wt{\cV}}$\,.
\end{lemma}
\begin{proof}
\smartqed
Let us prove that  $\pi_{\cV}$ is surjective on the level of cohomology. 

For this purpose, we observe that for every cocycle $v \in V$
its image $f(v)$ in $W$ is cohomologous to some cocycle of 
the form $\wt{f}(v')$, where $v'$ is a cocycle in $\wt{V}$. 
The latter follows easily from the fact that $f$ and $\wt{f}$ are 
quasi-isomorphisms.

In other words, for every degree $n$ cocycle $v \in V$ there exists 
a degree $n$ cocycle $v' \in V$ and a degree $(n-1)$ vector $w \in W$ such that
\begin{equation}
\label{f-wtf-w}
f(v) - \wt{f}(v') - \pa w = 0\,.  
\end{equation}

Hence, $v + \bs w - v'$ is a cocycle in $\Cyl(f, \wt{f})$ such that 
$\pi(v + \bs w - v') = v$\,.

Let us now prove that  $\pi_{\cV}$ is injective on the level of cohomology. 

For this purpose, we observe that the cocycle 
condition for $v + \bs w + v'  \in \Cyl(f, \wt{f})$ is equivalent to
the three equations: 
\begin{equation}
\label{v-cocycle}
\pa v = 0\,,
\end{equation} 
\begin{equation}
\label{v-pr-cocycle}
\pa v' = 0\,,
\end{equation}
and
\begin{equation}
\label{v-w-v-pr}
f(v) + \wt{f}(v') - \pa w = 0\,.
\end{equation}

Therefore, for every cocycle $v + \bs w + v'  \in \Cyl(f, \wt{f})$, 
the vectors $v$ and $v'$ are cocycles in $V$ and $\wt{V}$, respectively, 
and the cocycles $f(v)$ and $-\wt{f}(v')$ in $W$ are cohomologous.

Hence, $v + \bs w + v'  \in \Cyl(f, \wt{f})$ is a cocycle 
and $v$ is exact then so is $v'$, i.e. there exist vectors 
$v_1 \in V$ and $v'_1 \in \wt{V}$ such that 
$$
v = \pa v_1\,, \qquad v' = \pa v'_1\,.
$$  

Subtracting the coboundary of $v_1 \oplus \bs 0 \oplus  v'_1$ from 
$v \oplus \bs w \oplus v' $ we get a cocycle in $\Cyl(f, \wt{f})$ of the form 
\begin{equation}
\label{new-cocycle}
0 \oplus \bs(w -f(v_1) - \wt{f}(v'_1)) \oplus 0
\end{equation}

Since $w -f(v_1) - \wt{f}(v'_1)$ is a cocycle on $W$ and $\wt{f}$ is 
a quasi-isomorphism, there exists 
a cocycle $\wt{v} \in \wt{V}$ and a vector $w_1 \in W$ such that 
\begin{equation}
\label{w-v1-vpr-1}
w -f(v_1) - \wt{f}(v'_1) - \wt{f} ( \wt{v} )  - \pa(w_1) = 0\,. 
\end{equation}

Hence the cocycle \eqref{new-cocycle} is the coboundary of 
$$
0 \oplus ( - \bs w_1) \oplus \wt{v} \in \Cyl(f, \wt{f})\,.
$$

Thus $\pi_{V}$ is indeed injective on the level of cohomology.

Switching the roles $V \leftrightarrow \wt{V} $,  $f \leftrightarrow \wt{f} $, 
and  $\pi_{V} \leftrightarrow \pi_{\wt{V}} $ we also prove the desired 
statement about $\pi_{\wt{V}}$\,.
\qed
\end{proof}

%
%
%
%

\end{document}